\documentclass[a4paper]{amsart}
\usepackage{amsmath,amsthm,amssymb,latexsym,epic,bbm,comment,mathrsfs}
\usepackage{graphicx,enumerate,stmaryrd,color}
\usepackage[all,2cell]{xy}
\xyoption{2cell}

\newtheorem{theorem}{Theorem}
\newtheorem{lemma}[theorem]{Lemma}
\newtheorem{corollary}[theorem]{Corollary}
\newtheorem{proposition}[theorem]{Proposition}

\newtheorem{example}[theorem]{Example}

\newtheorem{remark}[theorem]{Remark}
\usepackage[all]{xy}
\usepackage[active]{srcltx}
\usepackage[parfill]{parskip}
\usepackage{enumerate}

\newcommand{\Gtil}{\widetilde{G}_{\sigma_1}}

\begin{document}
\title[Simple $\mathfrak{sl}_2$-modules that are
torsion free $U(\mathfrak{h})$-modules of rank $1$]
{Simple $\mathfrak{sl}_2$-modules that are\\  
torsion free $U(\mathfrak{h})$-modules of rank $1$}


\author[D.~Grantcharov, L.~K{\v r}i{\v z}ka and V.~Mazorchuk]{
Dimitar Grantcharov, Libor K{\v r}i{\v z}ka and Volodymyr Mazorchuk}

\begin{abstract}
We provide an explicit classification of all simple 
$\mathfrak{sl}_2$-modules that are torsion free of rank $1$
over the Cartan subalgebra. We also establish a similar result for the
first Weyl algebra and for the Lie superalgebra $\mathfrak{osp}(1|2)$.
\end{abstract}

\maketitle

{\bf Keywords:} skew Laurent polynomials, rational function, simple module, 
Lie algebra, Lie superalgebra, Weyl algebra

{\bf MSC 2020:} Primary: 16S32 Secondary: 16S99; 17B10; 17B35

\section{Introduction and description of the results}\label{s0}

\subsection{Historical overview}\label{s0.1}

Classification of simple modules over a given algebra is a fundamental
problem in contemporary representation theory. Under some natural 
assumptions, this problem is ``relatively easy'', for example, if one
works with a finite dimensional associative algebra over an algebraically
closed field. However, if the algebra in question is infinite dimensional,
this problem is usually ``hard'' or ``very hard''.

Classical examples of infinite dimensional algebras are the universal 
enveloping algebras of Lie algebras. The problem of classification of 
simple modules over finite dimensional complex Lie algebras is wide open.
In this statement, the smallest Lie algebra $\mathfrak{sl}_2$ comes
with a disclaimer. The problem of classification of all
simple $\mathfrak{sl}_2$-modules splits naturally into the following 
two subproblems:
\begin{itemize}
\item classification of all weight modules with respect to a Cartan subalgebra;
\item classification of all torsion free modules with respect to the same
Cartan subalgebra.
\end{itemize}
Classification of weight $\mathfrak{sl}_2$-modules is by now very well understood.
It goes back to some unpublished notes by Gabriel (as mentioned in \cite{Di96}) 
and can be found in  \cite{Dr83,Di96,Ma10}. The  papers \cite{Bl79,Bl81} 
establish a result which reduces classification of simple $\mathfrak{sl}_2$-modules 
that are torsion free with respect to a Cartan subalgebra to the problem of
describing equivalence classes of irreducible elements of a certain 
non-commutative Euclidean domain. This is, theoretically, very nice. However,
this result is in no way explicit.

Various explicit results available in the literature cover some special cases,
for example, the case of Whittaker modules, see \cite{AP}, or the case of
modules which are free of rank one over the Cartan subalgebra, see \cite{Ni}.
The recent paper \cite{GNZ} studies and gives an explicit description for a family
of simple $\mathfrak{sl}_2$-modules that are free of rank two over the
Cartan subalgebra.

\subsection{Description of the results}\label{s0.2}

In the present paper we provide an explicit classification of 
simple $\mathfrak{sl}_2$-modules (over the complex numbers) that are 
torsion free of rank one over the Cartan subalgebra. The classification is given in
terms of the following combinatorial parameters:
\begin{itemize}
\item a complex number $\vartheta$ which describes the central character
of our module;
\item another non-zero complex number which describes the leading term
of a certain polynomial;
\item and a function from a certain subset of complex numbers
(a strip in which the real part varies within a certain half-interval)
to the set of integers.
\end{itemize}
The latter function encodes a monic rational function.

Our main result, Theorem~\ref{thm-main}, provides an explicit 
description of the basis in our module, which is realized inside
the field of rational functions over the Cartan subalgebra.

As a bonus, we also obtain a similar result for the first Weyl algebra
and for the Lie superalgebra $\mathfrak{osp}(1\vert 2)$.

\subsection{Methods}\label{s0.3}

Our approach follows the theoretical description of 
\cite{Bl79,Bl81}, see also \cite{Ba} for a generalization of
this description to the so-called {\em Generalized Weyl Algebras}.
For a fixed central character $\vartheta$, there is a classical 
way to realize the corresponding primitive quotient
$U(\mathfrak{sl}_2)/(\mathtt{c}-\vartheta)$, where
$\mathtt{c}$ is the Casimir element, inside the 
skew Laurent polynomial ring $R=\mathbb{C}(\mathtt{h})[x,x^{-1},\sigma]$,
where the automorphism $\sigma$ of  $\mathbb{C}(\mathtt{h})$ is defined
by sending $\mathtt{h}$ to  $\mathtt{h}-2$.

The ring $R$ is a non-commutative Euclidean domain. $R$-modules of 
rank $1$ correspond to $R$-module structure on $\mathbb{C}(\mathtt{h})$
and hence are classified by certain equivalence classes of 
elements in $\mathbb{C}(\mathtt{h})$. We give full details for this 
classification in Section~\ref{s1}. 

From the general theory, see \cite{Bl79,Bl81,Ma10}, we know that simple
$\mathfrak{sl}_2$-modules that are torsion free of rank $1$ over
$\mathbb{C}[\mathtt{h}]$ are exactly the simple socles of the 
simple $R$-modules described in the previous paragraph.
Hence the final step of our  classification is an explicit 
determination of these simple socles. This requires some 
clever choices as well as a certain amount of care.

\subsection{Structure of the paper}\label{s0.4}

In Section~\ref{s1}, we collect all the preliminaries to
provide an explicit  classification of simple rank
$1$ modules over the skew Laurent polynomial algebra
$\mathbb{C}(\mathtt{h})[x,x^{-1},\sigma]$.
In Section~\ref{s2}, we apply this classification to 
classify simple  $\mathfrak{sl}_2$-modules that are 
torsion free of rank $1$ over the Cartan subalgebra.
In Section~\ref{s3}, we obtain a similar result for
the first Weyl algebra. Finally, in Section~\ref{s4}, 
we obtain a similar result for
the Lie superalgebra $\mathfrak{osp}(1|2)$.

\subsection*{Acknowledgements}
The first author is partially supported by Simons Collaboration Grant 855678. The third author is partially supported by the Swedish Research Council.
We thank the organizers of the ``Representation Theorem on Ice'' conference
that was held in Link{\"o}ping in January-February 2026,
which inspired our collaboration. 
\vspace{5mm}

\section{Modules over skew Laurent polynomial algebras}\label{s1}

\subsection{Skew Laurent polynomial algebras over fields}\label{s1.1}

Let $\Bbbk$ be a field and $\sigma$ an automorphism of $\Bbbk$.
The corresponding {\em skew Laurent polynomial algebra} $\Bbbk[x,x^{-1},\sigma]$
is the $\Bbbk$-algebra with basis $\{x^i\,:\,i\in\mathbb{Z}\}$ in which the
multiplication is defined via
\begin{displaymath}
x^ix^j=x^{i+j}\quad \text{ and }\quad
xr=\sigma(r)x,
\end{displaymath}
for all $i,j\in\mathbb{Z}$ and $r\in \Bbbk$.  

\subsection{The $\sigma$-subgroup of $\Bbbk^{\times}$}\label{s1.2}

Let $\Bbbk$ be a field and $\sigma$ an automorphism of $\Bbbk$.
Denote by $G_\sigma$ the set of all elements in $\Bbbk$
of the form $\frac{r}{\sigma(r)}$, where $r\in\Bbbk^{\times}$.

\begin{lemma}\label{lem1}
The subset $G_\sigma$ is, in fact, a subgroup of $\Bbbk^{\times}$. 
\end{lemma}

\begin{proof}
Clearly, $1=\frac{1}{\sigma(1)}\in G_\sigma$. Further,
since  $\frac{r}{\sigma(r)}\cdot \frac{s}{\sigma(s)}=\frac{rs}{\sigma(rs)}$,
for all $r,s\in\Bbbk$, it follows that $G_\sigma$ is closed under multiplication.
Finally, we have $\left(\frac{r}{\sigma(r)}\right)^{-1}=\frac{r^{-1}}{\sigma(r^{-1})}$ 
and hence $G_\sigma$ is closed under taking the inverse element. The claim follows. 
\end{proof}

We will call $G_\sigma$ the {\em $\sigma$-subgroup of $\Bbbk^{\times}$}.
Note that $G_\sigma$ acts on $\Bbbk$
(and hence also on $\Bbbk^{\times}$) by multiplication.

\begin{example}\label{ex2}
If $\sigma=\mathrm{id}_\Bbbk$, then $G_\sigma=\{1\}$. 
\end{example}

\begin{example}\label{ex3}
If   $\Bbbk=\mathbb{C}$ and $\sigma$ is the complex conjugation
$\sigma(a+b\mathbf{i})=a-b\mathbf{i}$, for $a,b\in\mathbb{R}$, then 
$G_\sigma=\{z\in \mathbb{C}\,:\, |z|=1\}$. Indeed, we have
\begin{displaymath}
\left|\frac{a+b\mathbf{i}}{a-b\mathbf{i}}\right| = 1
\end{displaymath}
and hence $G_\sigma\subset \{z\in \mathbb{C}\,:\, |z|=1\}$.
On the other hand, if $|z|=1$, then we have $z=\cos(\varphi)+\sin(\varphi)\mathbf{i}$,
for some $\varphi\in [0,2\pi)$, and thus
\begin{displaymath}
z=\frac{\cos(\frac{1}{2}\varphi)+\sin(\frac{1}{2}\varphi)\mathbf{i}}
{\cos(\frac{1}{2}\varphi)-\sin(\frac{1}{2}\varphi)\mathbf{i}}\in G_\sigma. 
\end{displaymath}
\end{example}

\subsection{Another example}\label{s1.3}

Let $\mathbb{L}$ be an algebraically closed field and $a$ a fixed
element of $\mathbb{L}^\times$. Consider $\Bbbk=\mathbb{L}(x)$, the field of rational functions,  and let
$\sigma$ be defined as the identity on $\mathbb{L}$ and, additionally, 
via $x\mapsto x+a$.

As $\mathbb{L}$ is algebraically closed, any element 
$\alpha\in \mathbb{L}(x)^\times$ can be uniquely written in the form
\begin{displaymath}
c_\alpha\prod_{t\in \mathbb{L}} (x-t)^{\mathtt{m}_\alpha(t)},
\end{displaymath}
where $c_\alpha\in \mathbb{L}^\times$ and $\mathtt{m}_\alpha
:\mathbb{L}\to \mathbb{Z}$
is a function such that 
$\mathtt{m}_\alpha(t)=0$ for all but finitely many values of $t$.
We will call $c_\alpha$ the {\em absolute value} and
$\mathtt{m}_\alpha$ the {\em exponent function}.

For a function $\mathtt{m}:\mathbb{L}\to \mathbb{Z}$,
the set of all $t\in \mathbb{L}$ for which $\mathtt{m}(t)\neq 0$
will be called  the {\em support} of $\mathtt{m}$.  Then we
can reformulate our condition on $\mathtt{m}$ in the previous
paragraph by requiring that 
the support of $\mathtt{m}$ is finite.

As the exponents behave additively under multiplication,
it follows that, mapping $\alpha$ to $(c_\alpha,\mathtt{m}_\alpha)$ 
defines an isomorphism between
the multiplicative group $\mathbb{L}(x)^\times$
and the direct product of the multiplicative
group $\mathbb{L}^\times$ with the additive group 
$\mathrm{Fsfun}(\mathbb{L},\mathbb{Z})$ of all
functions from $\mathbb{L}$ to $\mathbb{Z}$ with finite support.

Denote by $\mathrm{Fsfun}_\sigma(\mathbb{L},\mathbb{Z})$ the set
of all $\mathtt{m}\in\mathrm{Fsfun}(\mathbb{L},\mathbb{Z})$
satisfying the following condition: for any $t\in\mathbb{L}$,
we have
\begin{displaymath}
\sum_{i\in \mathbb{Z}}\mathtt{m}(t+ia) =0. 
\end{displaymath}
In other words, this condition says that the sum of the values
of $\mathtt{m}$ over any $\langle\sigma\rangle$-orbit in $\mathbb{L}$
is zero.
It is easy to see that $\mathrm{Fsfun}_\sigma(\mathbb{L},\mathbb{Z})$
is a subgroup of $\mathrm{Fsfun} (\mathbb{L},\mathbb{Z})$.
It is also easy to see that $\mathrm{Fsfun}_\sigma(\mathbb{L},\mathbb{Z})$
is generated, as an abelian group, by the functions $\mathtt{d}_s$,
where $s\in \mathbb{L}$, defined as follows
\begin{displaymath}
\mathtt{d}_s(t)=
\begin{cases}
1, & t=s;\\
-1, & t=s-a;\\
0,& \text{else}.
\end{cases}
\end{displaymath}

\begin{proposition}\label{prop4}
The restriction defines an isomorphism between $G_\sigma$
and $\mathrm{Fsfun}_\sigma(\mathbb{L},\mathbb{Z})$.
\end{proposition}

\begin{proof} 
Clearly, for any $\alpha\in\mathbb{L}(x)^\times$, the absolute value
of $\frac{\alpha}{\sigma(\alpha)}$ equals $1$. Therefore the 
isomorphism $\mathbb{L}(x)^\times \to \mathbb{L}^\times \times \mathrm{Fsfun}(\mathbb{L},\mathbb{Z})$ restricts to an embedding of $G_\sigma$ into
$\mathrm{Fsfun}(\mathbb{L},\mathbb{Z})$. Also,
for any $\alpha\in\mathbb{L}(x)^\times$ and any $t\in\mathbb{L}$,
setting $\beta:=\frac{\alpha}{\sigma(\alpha)}$, we have
\begin{displaymath}
\sum_{i\in \mathbb{Z}}\mathtt{m}_\beta(t+ia) =
\sum_{i\in \mathbb{Z}}\mathtt{m}_\alpha(t+ia)-
\sum_{i\in \mathbb{Z}}\mathtt{m}_{\sigma(\alpha)}(t+ia)=
\sum_{i\in \mathbb{Z}}\mathtt{m}_\alpha(t+ia)-
\sum_{i\in \mathbb{Z}}\mathtt{m}_\alpha(t+ia)=
0,
\end{displaymath}
where the equality
\begin{displaymath}
\sum_{i\in \mathbb{Z}}\mathtt{m}_{\sigma(\alpha)}(t+ia)=
\sum_{i\in \mathbb{Z}}\mathtt{m}_\alpha(t+ia)
\end{displaymath}
follows from the definition of $\sigma$. Therefore
the restriction map embeds $G_\sigma$
into $\mathrm{Fsfun}_\sigma(\mathbb{L},\mathbb{Z})$.

It remains to prove that this embedding is also surjective.
Since $\mathrm{Fsfun}_\sigma(\mathbb{L},\mathbb{Z})$ is generated
by  $\mathtt{d}_s$, where $s\in \mathbb{L}$, it is enough to
show that each $\mathtt{d}_s$ belongs to the image.
Clearly, $\mathtt{d}_s$ is the image of $\frac{x-s}{\sigma(x-s)}$.
The claim follows.
\end{proof}

\subsection{$\Bbbk[x,x^{-1},\sigma]$-modules of dimension $1$}\label{s1.4}

Let $R=\Bbbk[x,x^{-1},\sigma]$ be a skew Laurent polynomial algebra.
Here we are interested in classification of simple $R$-modules
that have dimension $1$ over $\Bbbk$. 

\begin{lemma}\label{lem5}
An $R$-module structure on $\Bbbk$ is uniquely determined 
by an element $u\in \Bbbk^{\times}$ via $x\cdot 1=u$.
\end{lemma}

\begin{proof}
If $x\cdot 1=u$, for some  $u\in \Bbbk^{\times}$, then, for any
$b\in \Bbbk$, we have $x\cdot b=x \cdot (b \cdot 1) = \sigma(b)u$. In particular,
$x\cdot{}_-$ is an invertible map on $\Bbbk$. Therefore
we can define the action of $x^{-1}$ as the inverse of the
action of $x$ and obtain an $R$-module structure on $\Bbbk$.

Conversely, given an $R$-module structure on $\Bbbk$, clearly,
$x\cdot 1$ must be non-zero. The claim follows.
\end{proof}

We will denote the above $R$-module structure on $\Bbbk$ by $N_u$.

\begin{proposition}\label{prop6}
For $u,v\in \Bbbk^{\times}$, the $R$-modules $N_u$ and $N_v$
are isomorphic if and only if $v\in G_\sigma\cdot u$.
\end{proposition}

\begin{proof}
Let $\varphi:N_u\to N_v$ be an $R$-module isomorphism.
Set $r:=\varphi(1)\in\Bbbk^{\times}$. Then $\varphi(b)=br$, for all $b\in\Bbbk$.
Then, for $b\in \Bbbk$, the equality
$\varphi(x\cdot b)=x\cdot \varphi(b)$ amounts to 
\begin{displaymath}
\sigma(b)ur=\sigma(b)\sigma(r)v. 
\end{displaymath}
If $b\neq 0$, we divide by $\sigma(b)$ to obtain
$ur=\sigma(r)v$, that is, $v=\frac{r}{\sigma(r)}u$.
In other words, $v\in G_\sigma\cdot u$.

Conversely, if $v\in G_\sigma\cdot u$, then 
$v=\frac{r}{\sigma(r)}u$ for some $r\in\Bbbk^{\times}$.
From the computation above, it follows that 
the map $\varphi:N_u\to N_v$ given by 
$\varphi(b)=br$, for all $b\in\Bbbk$,
intertwines the actions of $x$ on $N_u$ and $N_v$ and hence
is an $R$-module isomorphism between $N_u$ and $N_v$.
\end{proof}

\begin{corollary}\label{cor8}
The isomorphism classes of one-dimensional 
$R$-modules are in a natural bijection with $\Bbbk^{\times}/G_{\sigma}$.
\end{corollary}

We note that the group $\Bbbk^{\times}/G_{\sigma}$ is isomorphic to the 
first cohomology group $\mathrm{H}^1 (\langle \sigma \rangle, \Bbbk^{\times})$.

\section{Application to $\mathfrak{sl}_2$}\label{s2}

\subsection{General setup}\label{s2.1}

In this section, we work over $\mathbb{C}$
(or any other uncountable 
algebraically closed field of characteristic $0$ will also do). 
For a Lie algebra
$\mathfrak{a}$, we denote by $U(\mathfrak{a})$ the universal enveloping
algebra of $\mathfrak{a}$.

\subsection{The Lie algebra $\mathfrak{sl}_2$}\label{s2.2}

We will use \cite{Ma10} as a general reference for the Lie algebra
$\mathfrak{sl}_2$ and its representations.

Denote by $\mathfrak{g}$ the Lie algebra $\mathfrak{sl}_2$ with the 
standard basis
\begin{displaymath}
\mathtt{e}:=\left(\begin{array}{cc}0&1\\0&0\end{array}\right),\quad
\mathtt{f}:=\left(\begin{array}{cc}0&0\\1&0\end{array}\right),\quad
\mathtt{h}:=\left(\begin{array}{cc}1&0\\0&-1\end{array}\right).
\end{displaymath}
The Lie bracket on $\mathfrak{g}$ is then uniquely determined by
\begin{displaymath}
[\mathtt{h},\mathtt{e}]=2\mathtt{e},\quad 
[\mathtt{h},\mathtt{f}]=-2\mathtt{f},\quad 
[\mathtt{e},\mathtt{f}]=\mathtt{h}. 
\end{displaymath}
Consequently, in $U(\mathfrak{g})$, we have 
\begin{equation}\label{eq1}
\mathtt{h}\mathtt{e}=\mathtt{e}(\mathtt{h}+2)\quad
\text{ and }
\mathtt{h}\mathtt{f}=\mathtt{f}(\mathtt{h}-2). 
\end{equation}
We also have the Casimir element
\begin{equation}\label{eq1a}
\mathtt{c}:=(\mathtt{h}+1)^2+4 \mathtt{f}\mathtt{e}=
(\mathtt{h}-1)^2+4 \mathtt{e}\mathtt{f}\in U(\mathfrak{g}).
\end{equation}

For $\vartheta\in\mathbb{C}$, we denote by $U_\vartheta$
the quotient $U(\mathfrak{g})/U(\mathfrak{g})(\mathtt{c}-\vartheta)$.

\subsection{$U_\vartheta$ and skew Laurent polynomial algebra}\label{s2.3}

Consider the field $\Bbbk:=\mathbb{C}(\mathtt{h})$ of rational functions and its
automorphism $\sigma$ given by the identity on $\mathbb{C}$ and,
additionally, by sending $\mathtt{h}$ to $\mathtt{h}-2$. 
We can now consider the corresponding skew Laurent polynomial
algebra $R:=\Bbbk[x,x^{-1},\sigma]$.

For a fixed $\vartheta\in\mathbb{C}$, mapping
$\mathtt{e}$ to $x$ and $\mathtt{f}$ to 
\begin{displaymath}
\frac{\vartheta-(\mathtt{h}+1)^2}{4}x^{-1}
\end{displaymath}
gives rise to an injective algebra homomorphism $\Phi$ from
$U_\vartheta$ to $R$
(see \cite[Theorem~6.2.12]{Ma10}). Note that $\Phi(\mathtt{h})=\mathtt{h}$.
Pulling back via $\Phi$ defines the restriction  functor
\begin{displaymath}
\mathrm{Res}^{R}_{U_\vartheta}:R\text{-Mod}\to  U_\vartheta\text{-Mod}.
\end{displaymath}
In the remainder of this section, we will identify 
$U_\vartheta$ with its image under $\Phi$.

Let us recall that, as a $\mathbb{C}$-vector space, the algebra
$\mathbb{C}(\mathtt{h})$  has a basis consisting of  all 
monomials  $\mathtt{h}^i$, where $i\geq 0$, as well as  
all fractions of the form $\frac{1}{(\mathtt{h}-t)^m}$, where $t\in \mathbb{C}$
and $m\in\mathbb{Z}_{>0´}$. We will denote this basis by $\mathbf{B}$.

\subsection{Simple $R$-modules and the corresponding
simple $U_\vartheta$-modules}\label{s2.4}

Recall, from Subsection~\ref{s1.4}, that simple $R$-modules
of dimension $1$ over $\Bbbk$ are of the form $N_u$, for some
$u\in \Bbbk$. The action of $x$ on $N_u$ is given by $x\cdot b=\sigma(b)u$,
for $b\in\Bbbk$. This means that 
\begin{displaymath}
\mathtt{e}\cdot b=\sigma(b)u,\quad\text{ for }\quad b\in\Bbbk.
\end{displaymath}
Consequently, we also have  
\begin{displaymath}
\mathtt{f}\cdot b=
\frac{(\vartheta-(\mathtt{h}+1)^2)}{4}\cdot
\frac{\sigma^{-1}(b)}{\sigma^{-1}(u)},\quad\text{ for }\quad b\in\Bbbk.
\end{displaymath}
Also, $N_u\cong N_v$ if and only if
$u\in G_\sigma v$. Given such $N_u$, the corresponding 
$U_\vartheta$-module $\mathrm{Res}^{R}_{U_\vartheta}(N_u)$
has a unique simple submodule, see \cite[Lemma~6.3.8]{Ma10}.
We denote this simple $U_\vartheta$-module by $L_u$.

From Subsection~\ref{s1.3} we know that the group $G_\sigma$
is isomorphic to $\mathrm{Fsfun}_\sigma(\mathbb{L},\mathbb{Z})$.
For a fixed $\omega\in\mathbb{R}$,
denote by $\mathbb{C}_\omega$ the set of all $z\in\mathbb{C}$ whose real
part belongs to the half-interval $[\omega,\omega+2)$. Then $\mathbb{C}_\omega$
is a cross-section of the (additive) action of $2\mathbb{Z}$ on $\mathbb{C}$
(note that our $\sigma$ is defined in terms of this action). Then,
each coset $\zeta\in \Bbbk^{\times}/\mathrm{Fsfun}_\sigma(\mathbb{C},\mathbb{Z})$
contains a unique element $u\in \Bbbk^{\times}$ of the form
\begin{equation}\label{eq11a}
c_u\prod_{t\in \mathbb{C}_\omega} (\mathtt{h}-t)^{\mathtt{m}_u(t)}.
\end{equation}
Consequently, by Corollary~\ref{cor8}, the isomorphism classes of 
simple $R$-module structures on $\mathbb{C}(x)$ are in bijection 
with pairs $(c,\mathtt{m})$, where $c\in\mathbb{C}^{\times}$ and
$\mathtt{m}\in \mathrm{Fsfun}(\mathbb{C}_\omega,\mathbb{Z})$.
Therefore, by \cite[Theorem~6.3.6]{Ma10}, the set of all
$L_u$, for $u$ given by Equation~\ref{eq11a}, is a complete and
irredundant set of representatives of the isomorphism classes of
simple $U_\vartheta$-modules which are torsion free of rank one
over $U(\mathtt{h})$.

We will make the following choice of $\omega$, depending on
$\vartheta$: let $r_1$ and $r_2$ be the two roots
of $\vartheta-(\mathtt{h}+1)^2$ (possibly, $r_1=r_2$). We set 
$\omega$ to be $3$ plus the maximum between the real parts of
$r_1$ and $r_2$. For example, if $\vartheta=0$, then 
$r_1=r_2=-1$ and $\omega=2$. In this way, we ensure that 
the real parts of both $r_1$ and $r_2$ are strictly less than
and even ``far away'' from $\omega$. This will simplify 
some of the results and also some of the arguments to follow.

\subsection{$L_u$ vs $\mathbb{C}[\mathtt{h}]$}\label{s2.5}

Let us now describe the structure of $L_u$ in more detail.
For this we assume that 
\begin{equation}\label{eq11}
u=c\prod_{t\in \mathbb{C}_\omega} (\mathtt{h}-t)^{\mathtt{m}(t)},
\end{equation}
for some $c\in\mathbb{C}^{\times}$ and $\mathtt{m}\in \mathrm{Fsfun}(\mathbb{C}_\omega,\mathbb{Z})$. 
We set
\begin{equation}\label{eq11bc}
\alpha(\mathtt{h})=\prod_{t\in \mathbb{C}_\omega:
\mathtt{m}(t)>0} (\mathtt{h}-t)^{\mathtt{m}(t)}\qquad\text{ and }\qquad
\beta(\mathtt{h})=\prod_{t\in \mathbb{C}_\omega:
\mathtt{m}(t)<0} (\mathtt{h}-t)^{-\mathtt{m}(t)}.
\end{equation}
Then we have $u=c\frac{\alpha(\mathtt{h})}{\beta(\mathtt{h})}$
with $\mathrm{g.c.d.}(\alpha(\mathtt{h}),\beta(\mathtt{h}))=1$.

Consider the set $L_u\cap\mathbb{C}[\mathtt{h}]$. This is a non-zero ideal
of $\mathbb{C}[\mathtt{h}]$ as any non-zero $\mathbb{C}[\mathtt{h}]$-sub\-mo\-dule
of $\mathbb{C}(\mathtt{h})$ intersects $\mathbb{C}[\mathtt{h}]$ non-trivially.
As $\mathbb{C}[\mathtt{h}]$ is a PID, this ideal is generated by a 
uniquely determined monic non-zero polynomial $f_u(\mathtt{h})$.

Recall that $r_1$ and $r_2$ denote the two roots
of $\vartheta-(\mathtt{h}+1)^2$ (we have $r_1=r_2$ if and only if $\vartheta=0$).

\begin{proposition}\label{prop-s1.5-1}
We have  $f_u(\mathtt{h})=1$ and hence 
$\mathbb{C}[\mathtt{h}]\subset L_u$.
\end{proposition}

\begin{proof}
For any $v(\mathtt{h})\in \mathbb{C}[\mathtt{h}]$, we have 
$\beta(\mathtt{h})\mathtt{e}\cdot v(\mathtt{h})=c v(\mathtt{h}-2)\alpha(\mathtt{h})
\in \mathbb{C}[\mathtt{h}]$. In particular,
$cf_u(\mathtt{h}-2)\alpha(\mathtt{h})$ must be divisible by 
$f_u(\mathtt{h})$. Applying $\beta(\mathtt{h})\mathtt{e}$ again and again, we get that 
$f_u(\mathtt{h}-2i)\alpha(\mathtt{h}-2(i-1))\alpha(\mathtt{h}-2(i-2))\cdots \alpha(\mathtt{h})$
must be divisible by $f_u(\mathtt{h})$, for all $i\in\mathbb{Z}_{>0}$. For some $i\gg 0$, we have 
$\mathrm{g.c.d.}(f_u(\mathtt{h}-2i),f_u(\mathtt{h}))=1$. This implies
that $f_u(\mathtt{h})$ divides the corresponding
$\alpha(\mathtt{h}-2(i-1))\alpha(\mathtt{h}-2(i-2))\cdots \alpha(\mathtt{h})$.
Hence all roots of $f_u(\mathtt{h})$ belong to the set 
\begin{equation}\label{eq-12n1}
\{t\in\mathbb{C}_\omega\,:\,\mathtt{m}(t)>0\}+2\mathbb{Z}_{\geq 0}.
\end{equation}

For any $v(\mathtt{h})\in \mathbb{C}[\mathtt{h}]$, we have 
\begin{displaymath}
\alpha(\mathtt{h}+2)\mathtt{f}\cdot v(\mathtt{h})=\frac{1}{4c} 
v(\mathtt{h}+2)\beta(\mathtt{h}+2)(\vartheta-(\mathtt{h}+1)^2). 
\end{displaymath}
In particular, $f_u(\mathtt{h}+2)\beta(\mathtt{h}+2)(\vartheta-(\mathtt{h}+1)^2)$
must be divisible by 
$f_u(\mathtt{h})$.
Applying $\alpha(\mathtt{h}+2)\mathtt{f}$ again and again, we get that 
\begin{multline*}
f_u(\mathtt{h}+2i)\cdot
\beta(\mathtt{h}+2i)
\beta(\mathtt{h}+2(i-1))\cdots 
\beta(\mathtt{h}+2)\cdot \\ \cdot
(\vartheta-(\mathtt{h}+2(i-1)+1)^2)
(\vartheta-(\mathtt{h}+2(i-2)+1)^2)\cdots 
(\vartheta-(\mathtt{h}+1)^2)
\end{multline*}
must be divisible by $f_u(\mathtt{h})$.
For some $i\gg 0$, we have 
$\mathrm{g.c.d.}(f_u(\mathtt{h}+2i),f_u(\mathtt{h}))=1$. This implies
that $f_u(\mathtt{h})$ divides the product of the other terms.
Therefore any root of
$f_u(\mathtt{h})$ belongs to the set
\begin{equation}\label{eq-eq23}
\left(\{r_1,r_2\}\cup \{t\in\mathbb{C}_\omega\,:\,\mathtt{m}(t)<0\}\right)
-2\mathbb{Z}_{\geq 0}.
\end{equation}
Due to our choice of $\omega$, the sets in Equations~\eqref{eq-12n1}
and \eqref{eq-eq23} are disjoint. Consequently, $f_u(\mathtt{h})=1$ and we are done.
\end{proof}

\subsection{Explicit description}\label{s2.6}

Recall that $\vartheta-(\mathtt{h}+1)^2=-(\mathtt{h}-r_1)(\mathtt{h}-r_2)$.
Let $t_1,t_2\in \mathbb{C}_\omega$ 
and $n_1,n_2\in \mathbb{Z}_{>0}$ be such that
$t_1-r_1=2n_1$ and $t_2-r_2=2n_2$. All this is well-defined due to our
choice of $\omega$, moreover, we have $n_1,n_2\geq 2$. 
We also denote by $m_1$ and $m_2$ the multiplicity of
$r_1$ and $r_2$ as a root of $\vartheta-(\mathtt{h}+1)^2$, respectively.

Let $K_u$ denote the subspace of $N_u$ spanned by the following elements
\begin{itemize}
\item $\mathbb{C}[\mathtt{h}]$;
\item $\frac{1}{(\mathtt{h}-s-2i)^{k}}$, 
where $i\in\mathbb{Z}_{\geq 0}$ and $1\leq k\leq -\mathtt{m}(s)$,
for each $s\in \mathbb{C}_\omega$ such that $\mathtt{m}(s)<0$;
\item $\frac{1}{(\mathtt{h}-s+2i)^{k}}$, 
where $i\in\mathbb{Z}_{>0}$ and 
\begin{displaymath}
1\leq k\leq \mathtt{m}(s)-
|\{j\in\{1,2\}\,:\, r_j\in\{s-4,s-6,\dots,s-2i\}\}|. 
\end{displaymath}
for each $s\in \mathbb{C}_\omega$ such that $\mathtt{m}(s)>0$.
\end{itemize}

\begin{theorem}\label{thm-main}
We have $K_u=L_u$.
\end{theorem}

\begin{proof}
We know from Proposition~\ref{prop-s1.5-1} that 
$\mathbb{C}[\mathtt{h}]\subset L_u$. Therefore we only
need to prove that $U_\vartheta\mathbb{C}[\mathtt{h}]=K_u$.
For simplicity, in the remainder of the proof we will 
talk about elements in $N_u$ only up to non-zero scalars.

Take $s\in \mathbb{C}_\omega$ such that $\mathtt{m}(s)<0$.
Applying $\mathtt{e}$ to $1$ gives $\frac{\alpha(\mathtt{h})}{\beta(\mathtt{h})}$
which has $(\mathtt{h}-s)^{\mathtt{m}(s)}$ as a factor. Note that the numerator
is coprime with $\mathtt{h}-s$, hence, writing 
$\frac{\alpha(\mathtt{h})}{\beta(\mathtt{h})}$ with respect to the basis
$\mathbf{B}$ gives $\frac{1}{(\mathtt{h}-s)^{-\mathtt{m}(s)}}$ with some
non-zero coefficient. Applying $\frac{\beta(\mathtt{h})}{(\mathtt{h}-s)}$
gives therefore $\frac{1}{(\mathtt{h}-s)}$ plus a polynomial in $\mathtt{h}$.
Hence $\frac{1}{(\mathtt{h}-s)}\in L_u$. In case $-\mathtt{m}(s)\geq 2$,
applying, if necessary, $\frac{\beta(\mathtt{h})}{(\mathtt{h}-s)^2}$
to $\frac{\alpha(\mathtt{h})}{\beta(\mathtt{h})}$ and proceeding inductively
we get $\frac{1}{(\mathtt{h}-s)^{k}}\in L_u$, for all $k=1,2,\dots,-\mathtt{m}(s)$.
Applying to these elements $\mathtt{e}$ repeatedly and using arguments
similar to the above, we get 
$\frac{1}{(\mathtt{h}-s-2j)^{k}}\in L_u$, for all $k=1,2,\dots,-\mathtt{m}(s)$
and all $j\in\mathbb{Z}_{\geq 0}$.

Note that the application of $\mathtt{f}$ to $\frac{1}{(\mathtt{h}-s)^{k}}$, 
where $k\in\{1,2,\dots,-\mathtt{m}(s)\}$, does not produce any denominators
of the form $\frac{1}{(\mathtt{h}-s+2)^{l}}$ as these cancel with
the factor $(\mathtt{h}-s+2)^{-\mathtt{m}(s)}$ of 
$\beta(\mathtt{h}+2)$ in the numerator.

Next, take $s\in \mathbb{C}_\omega\setminus\{t_1,t_2\}$ such that $\mathtt{m}(s)>0$.
Swapping the roles of $\mathtt{e}$ and $\mathtt{f}$ in the arguments above,
one shows that $\frac{1}{(\mathtt{h}-s+2i)^{k}}\in L_u$, 
where $i\in\mathbb{Z}_{>0}$ and $1\leq k\leq \mathtt{m}(s)$.
Here the point is that the factor $\vartheta-(\mathtt{h}+1)^2$
in the numerator of the action of $\mathtt{f}$ remains coprime with the
denominators and hence does not affect the argument. Also, we cannot 
obtain any $\frac{1}{(\mathtt{h}-s)^{l}}$ from 
$\frac{1}{(\mathtt{h}-s+2)^{k}}$ using $\mathtt{e}$ due to the
$(\mathtt{h}-s)^{\mathtt{m}(s)}$ factor in $\alpha(\mathtt{h})$. 
Finally, note the restriction $i\neq 0$ above, in contrast to the previous paragraph,
due to $\sigma^{-1}(\beta)$ in the formula for $\mathtt{f}$.

It remains to consider the case $s=t_j$, for some $j\in\{1,2\}$, and
$\mathtt{m}(s)>0$. Assume first that $t_1\neq t_2$.
If we try the approach in the previous paragraph,
the first step works just fine (due to our choice of $\omega$), 
however, when trying to apply 
$\mathtt{f}$ repeatedly, we eventually will come to a situation
when the factors of $\vartheta-(\mathtt{h}+1)^2$ in the numerator of
action of $\mathtt{f}$ will cancel some factors in
the denominator  given by $\alpha(\mathtt{h}-s+2j)$. The corresponding
step is given by $n_j$. The number of
factors in the numerator is $m_j=1$ and in the denominator it is
$\mathtt{m}(s)$. So, if $m_j=1\geq \mathtt{m}(s)$, we will not be able
to create any $\frac{1}{(\mathtt{h}-s+2n_j)^{l}}$ and thus must stop.
If $m_j=1<\mathtt{m}(s)$, we will be able
to create elements of the form $\frac{1}{(\mathtt{h}-s+2n_j)^{\mathtt{m}(s)-m_j}}$
and then proceed as in the previous paragraph.

Finally, let us consider the case $s=t_1=t_2$. Here we can have
$r_1=r_2=-1$ or $r_1-r_2\in2\mathbb{Z}_{>0}$. If
$r_1=r_2=-1$, then $t_1=t_2=3$. The first step of the above
approach works just fine, but at the second step we have
$(\mathtt{h}+1)^2$ in the numerator and hence we can only create
the exponent $\mathtt{m}(s)-2$ (if it is positive). The rest
is similar to the previous paragraph.

If $r_1-r_2\in2\mathbb{Z}_{>0}$, our maximal possible exponent 
that we can obtain will have
to decrease by $1$ twice: first at step $n_1$ and then at 
step $n_2$. The rest is similar to the above.
\end{proof}

We note that our choice of $\omega$ guarantees that the whole of 
$\mathbb{C}[h]$  belongs to $L_u$, simplifying some arguments in the
proof of the previous theorem. For the general choice of 
$\omega$, one could instead consider the $\mathfrak{sl}_2$ submodule
of $R$ generated by $\mathbb C [h]$ and then look for the simple 
$\mathfrak{sl}_2$-socle of that module.

\begin{example}\label{ex-11}
Let $\vartheta=0$ and $u=\mathtt{h}-2$. Then 
$\vartheta-(\mathtt{h}+1)^2=-(\mathtt{h}+1)^2$ has root $-1$ of multiplicity $2$.
Note that $2-(-1)\not\in2\mathbb{Z}$.

Applying $\mathtt{f}$ to $1$ gives $\frac{-\frac{1}{4}(\mathtt{h}+1)^2}{\mathtt{h}}$ and,
since the numerator and the denominator are coprime, we obtain
$\frac{1}{\mathtt{h}}\in L_u$. Applying $\mathtt{f}$ again and again,
we get $\frac{1}{\mathtt{h}+2i}\in L_u$, for all $i\geq 0$.

Applying $\mathtt{e}$ to $\frac{1}{\mathtt{h}}$ gives $1$
due to cancellation of $\mathtt{h}-2$,
so the linear span of $\mathbb{C}[\mathtt{h}]$
together with all  $\frac{1}{\mathtt{h}+2i}$, for  $i\geq 0$
is $\mathfrak{g}$-invariant. Hence this linear span is our $L_u$.
\end{example}

\begin{example}\label{ex-12}
Let $\vartheta=0$ and $u=\mathtt{h}-3$. Then 
$\vartheta-(\mathtt{h}+1)^2=-(\mathtt{h}+1)^2$ has root $-1$ of multiplicity $2$.
Note that $3-(-1)=2\cdot 2$.

Applying $\mathtt{f}$ to $1$ gives $\frac{-\frac{1}{4}(\mathtt{h}+1)^2}{\mathtt{h}-1}$ and,
since the numerator and the denominator are coprime, we obtain
$\frac{1}{\mathtt{h}-1}\in L_u$. Applying $\mathtt{f}$ again gives
$\frac{-\frac{1}{4}(\mathtt{h}+1)}{\mathtt{h}-1}$,
which is a linear combination of $1$ and $\frac{1}{\mathtt{h}-1}$. 
Therefore our $L_u$ is spanned by $\mathbb{C}[\mathtt{h}]$
and $\frac{1}{\mathtt{h}-1}$. Note that, in this case, $L_u$ is free over
$\mathbb{C}[\mathtt{h}]$ of rank one with basis $\frac{1}{\mathtt{h}-1}$.
\end{example}

\begin{example}\label{ex-13}
Let $\vartheta=0$ and $u=(\mathtt{h}-3)^3$. Then 
$\vartheta-(\mathtt{h}+1)^2=-(\mathtt{h}+1)^2$ has root $-1$ of multiplicity $2$.
Note that $3-(-1)=2\cdot 2$.

Applying $\mathtt{f}$ to $1$ gives $\frac{-\frac{1}{4}(\mathtt{h}+1)^2}{(\mathtt{h}-1)^3}$ and,
since the numerator and the denominator are coprime, we obtain
$\frac{1}{(\mathtt{h}-1)^3},\frac{1}{(\mathtt{h}-1)^2},
\frac{1}{\mathtt{h}-1}\in L_u$. Applying $\mathtt{f}$ 
to $\frac{1}{(\mathtt{h}-1)^3}$, we obtain 
\begin{displaymath}
\frac{-\frac{1}{4}(\mathtt{h}+1)^2}{(\mathtt{h}+1)^3(\mathtt{h}-1)^3}=
-\frac{1}{4(\mathtt{h}+1)(\mathtt{h}-1)^3}.
\end{displaymath}
From here, applying $-4(\mathtt{h}-1)^3$ gives $\frac{1}{\mathtt{h}+1}\in L_u$.
Proceeding similarly, we obtain that
$\frac{1}{\mathtt{h}+1+2i}\in L_u$, for all $i\in\mathbb{Z}_{\geq 0}$.
Therefore our $L_u$ is spanned by $\mathbb{C}[\mathtt{h}]$,
$\frac{1}{(\mathtt{h}-1)^3}$,
$\frac{1}{(\mathtt{h}-1)^2}$
and all $\frac{1}{\mathtt{h}-1+2i}$, for $i\in\mathbb{Z}_{\geq 0}$.
\end{example}

\begin{corollary}\label{cor}
The module $L_u$ is finitely generated over $\mathbb{C}[\mathtt{h}]$
if and only if, for any $s\in\mathbb{C}_\omega$, we have
$0\leq \mathtt{m}(s)\leq |\{j\in\{1,2\}\,:\, r_j\in s-2\mathbb{Z}_{>0}\}|$.
\end{corollary}

\begin{proof}
The module $L_u$ is finitely generated over $\mathbb{C}[\mathtt{h}]$
if and only if the number of genuine fractions in the description of
$L_u$ is finite. Directly from the definitions it follows that
this is equivalent to the
condition of the corollary.
\end{proof}

Note that being both torsion-free and finitely generated 
over $\mathbb{C}[\mathtt{h}]$ is  equivalent to being free of finite rank.
Therefore Corollary~\ref{cor} classifies all 
simple $\mathfrak{sl}_2$-module that are free of rank one over
$\mathbb{C}[\mathtt{h}]$.

\section{Application to the first Weyl algebra}\label{s3}

\subsection{General setup}\label{s3.1}

In this section, we work over $\mathbb{C}$
(alternatively, an algebraically closed field of characteristic $0$). 
Recall that the {\em first Weyl algebra} $\mathbf{A}_1$
is given by the presentation $\mathbf{A}_1:=\langle a,b\,:\, ab-ba=1\rangle$.
The algebra $\mathbf{A}_1$ is a simple algebra.

Consider the same field $\Bbbk:=\mathbb{C}(\mathtt{h})$ of rational functions 
as in Section~\ref{s2} and the same  automorphism $\sigma$ 
of this field given by the identity on $\mathbb{C}$ and,
additionally, by sending $\mathtt{h}$ to $\mathtt{h}-2$. 
We can now consider the corresponding skew Laurent polynomial
algebra $R:=\Bbbk[x,x^{-1},\sigma]$.

Sending $a$ to $x$ and $b$ to $-\frac{\mathtt{h}}{2}x^{-1}$
defines an embedding $\Psi$ of $\mathbf{A}_1$ into $R$.
Hence any $R$-module can be considered as an $\mathbf{A}_1$-module
by pulling back via $\Psi$.

\subsection{Simple torsion free $\mathbf{A}_1$-modules of rank $1$}\label{s3.2}

By Corollary~\ref{cor8}, the isomorphism classes of 
simple $R$-module structures on $\mathbb{C}(x)$ are in bijection 
with pairs $(c,\mathtt{m})$, where $c\in\mathbb{C}^{\times}$ and
$\mathtt{m}\in \mathrm{Fsfun}(\mathbb{C}_0,\mathbb{Z})$
(note that we take $\omega=0$).

For $u$ given by Equation~\ref{eq11a}, we have the corresponding
simple $R$-module $N_u$ of rank $1$. We denote by $F_u$
the simple socle of $N_u$, when considered as an $\mathbf{A}_1$-module.
Consequently, the set of all these $F_u$, for $u$ given by Equation~\ref{eq11a}, 
is a complete and irredundant set of representatives of the isomorphism 
classes of simple $\mathbf{A}_1$-modules which are torsion free of rank one
over the subalgebra $\mathbb{C}[ab]$. We can now present an
explicit description of all such $F_u$. Recall that $\mathtt{h}=-2ba$.

\begin{lemma}\label{kem-3-31}
For all $u$ as above, we have $\mathbb{C}[\mathtt{h}]\subset F_u$. 
\end{lemma}

\begin{proof}
Mutatis mutandis the proof of Proposition~\ref{prop-s1.5-1}.
\end{proof}

Let $D_u$ denote the subspace of $N_u$ spanned by the following elements
\begin{itemize}
\item $\mathbb{C}[\mathtt{h}]$;
\item $\frac{1}{(\mathtt{h}-s-2i)^{k}}$, 
where $i\in\mathbb{Z}_{\geq 0}$ and $1\leq k\leq -\mathtt{m}(s)$,
for each $s\in \mathbb{C}_0$ such that $\mathtt{m}(s)<0$;
\item $\frac{1}{(\mathtt{h}-s+2i)^{k}}$, 
where $i\in\mathbb{Z}_{>0}$ and $1\leq k\leq \mathtt{m}(s)$,
for each $s\in \mathbb{C}_0\setminus\{0\}$ such that $\mathtt{m}(s)>0$.
\item $\frac{1}{(\mathtt{h}+2i)^{k}}$, 
where $i\in\mathbb{Z}_{>0}$ and 
$1\leq k\leq \mathtt{m}(0)-1$,  if $\mathtt{m}(0)>0$.
\end{itemize}

\begin{theorem}\label{thm-main-2}
We have $D_u=F_u$. Furthermore, the module $D_u$, for $u$ as above,
for a complete and irredundant list of representatives of the isomorphism
classes of simple $\mathbf{A}_1$-modules that are torsion-free or rank
one over the subalgebra $\mathbb{C}[ab]$.
\end{theorem}

\begin{proof}
Mutatis mutandis the proof of Theorem~\ref{thm-main}. 
\end{proof}

\begin{corollary}\label{cor-2}
The module $F_u$ is finitely generated over $\mathbb{C}[\mathtt{h}]$
if and only if $\mathtt{m}$ is either the zero function or
$\mathtt{m}(s)\neq 0$ implies $s=0$ and, additionally,
$\mathtt{m}(0)=1$.
\end{corollary}

\section{Application to $\mathfrak{osp}(1|2)$}\label{s4}

\subsection{Setup, SCasimir, and the oscillator quotient}\label{s4.1}
In this subsection we consider an application of our approach to yet another
situation, namely,  the Lie superalgebra $\mathfrak s=\mathfrak{osp}(1|2)$.
The Block-style description of simple modules for this superalgebra
can be found in \cite{B000}.

Consider the Lie superalgebra $\mathfrak s=\mathfrak{osp}(1|2)$ with even generators $\mathtt e,\mathtt f,\mathtt h$
and odd generators $\mathtt p,\mathtt q$, satisfying
\[
[\mathtt h,\mathtt e]=2\mathtt e,\qquad [\mathtt h,\mathtt f]=-2\mathtt f,\qquad [\mathtt e,\mathtt f]=\mathtt h,
\]
\[
[\mathtt h,\mathtt p]=\mathtt p,\qquad [\mathtt h,\mathtt q]=-\mathtt q,\qquad
\mathtt p^2=\mathtt e,\qquad \mathtt q^2=-\mathtt f,\qquad \mathtt p\mathtt q+\mathtt q\mathtt p=\mathtt h.
\]
Set
\begin{equation}\label{eq:s4-sigma}
\Sigma:=\mathtt p\mathtt q-\mathtt q\mathtt p+\frac12\in U(\mathfrak s),
\end{equation}
which sometimes is called the SCasimir element. As usual, we denote by 
$\Pi$ the parity change functor.

 Most of the structural results for $U(\mathfrak{osp}(1|2))$ used in this section can be found in the work of Pinczon~\cite{Pi90}. The following is a well-known  result that follows from the defining relations of $\mathfrak s$.
\begin{lemma}\label{lem:s4-sigma}
We have $\Sigma\mathtt p+\mathtt p\Sigma=0$ and $\Sigma\mathtt q+\mathtt q\Sigma=0$.
Consequently, $[\Sigma,\mathtt h]=[\Sigma,\mathtt e]=[\Sigma,\mathtt f]=0$ and $\Sigma^2$ is central in $U(\mathfrak s)$.
\end{lemma}

Recall that $\mathbf{A}_1=\langle a,b\,:\,ab-ba=1\rangle$ is the first Weyl algebra.
Define an algebra homomorphism $\Theta:U(\mathfrak s)\to \mathbf{A}_1$ by the assignment
\begin{equation}\label{eq:s4-theta}
\Theta(\mathtt p)=\frac{1}{\sqrt2}\,a,\qquad
\Theta(\mathtt q)=-\frac{1}{\sqrt2}\,b,
\end{equation}
and, hence, $\Theta(\mathtt h)=-\frac12(ab+ba)$,
$\Theta(\mathtt e)=\frac12\,a^2$, and 
$\Theta(\mathtt f)=-\frac12\,b^2$. The fact that $\Theta$ is a surjective homomorphism 
follows directly from the definition. We also have that $\Theta(\Sigma)=0$, so  $\Theta$ factors through $\overline U:=U(\mathfrak s)/U(\mathfrak s)\Sigma U(\mathfrak s)$.
\begin{proposition}\label{prop:s4-Ubar-A1}
The induced map $\overline{\Theta} : \overline U\to \mathbf{A}_1$ is an isomorphism.
\end{proposition}
\begin{proof} One easily checks that there is a homomorphism $\mathbf{A}_1 \to \overline U$ defined via
$a\mapsto \sqrt2\,\overline{\mathtt p}$ and $b\mapsto -\sqrt2\,\overline{\mathtt q}$ and that this homomorphism is the inverse of $\overline{\Theta}$.
\end{proof}

\medskip
\noindent\textbf{Torsion-free rank and superrank.}
Set $H:=U(\mathfrak h)=\mathbb{C}[\mathtt h]$ and $\Bbbk:=\mathbb{C}(\mathtt h)$ as before. We consider $\mathtt h$ (hence, $H$ and $\Bbbk$) as purely even.
If $X$ is an $H$-torsion free module, define its \emph{$H$-rank} by
\[
\mathrm{rk}_H(X):=\dim_{\Bbbk}\bigl(\Bbbk\otimes_H X\bigr).
\]
If $M=M_{\bar0}\oplus M_{\bar1}$ is a $\mathbb{Z}_2$-graded $H$-module, we say that $M$ is
\emph{$H$-torsion free of superrank $(1|1)$} if both $M_{\bar0}$ and $M_{\bar1}$ are $H$-torsion free and
\[
\dim_{\Bbbk}(\Bbbk\otimes_H M_{\bar0})=1,\qquad \dim_{\Bbbk}(\Bbbk\otimes_H M_{\bar1})=1.
\]

\subsection{Simple \emph{ungraded} $U(\mathfrak s)$-modules of $H$-rank $1$}\label{s4.2}

\begin{lemma}\label{lem:s4-ungraded-sigma0}
Let $M$ be a (not necessarily $\mathbb Z_2$-graded) $U(\mathfrak s)$-module which is $H$-torsion free of rank $1$.
Then $\Sigma$ acts by $0$ on $M$.
\end{lemma}

\begin{proof}
Localize $M$ to $M_{\Bbbk}:=\Bbbk\otimes_{H}M\cong \Bbbk$ and pick $v\neq 0$
in $M_{\Bbbk}$. Then $M_{\Bbbk}=\Bbbk v$.
From $[\mathtt h,\mathtt p]=\mathtt p$ and $[\mathtt h,\mathtt q]=-\mathtt q$,
there exist $u,w\in \Bbbk^{\times}$ such that
\[
\mathtt p\cdot(bv)=b(\mathtt h-1)\,u(\mathtt h)\,v,\qquad
\mathtt q\cdot(bv)=b(\mathtt h+1)\,w(\mathtt h)\,v,\qquad \text{ for any } b\in \Bbbk.
\]
Since $[\Sigma,\mathtt h]=0$ by Lemma~\ref{lem:s4-sigma}, 
$\Sigma$ acts on $M_{\Bbbk}$ as multiplication by some $\varphi(\mathtt h)\in \Bbbk$.
Using $\Sigma\mathtt p+\mathtt p\Sigma=0$ gives $\varphi(\mathtt h)+\varphi(\mathtt h-1)=0$.
If $\varphi\neq 0$, then $\varphi(\mathtt h-n)=(-1)^n\varphi(\mathtt h)$, for all $n\ge 0$, forcing infinitely many translates of any zero and of  any pole,
impossible for a rational function. Hence $\varphi=0$.
\end{proof}

Pulling back along $\Theta$ 
(which we denote by $\Theta^*$, as usual) induces a bijection between the
isomorphism classes of simple $\mathbf{A}_1$-modules which are torsion free of rank $1$ over the subalgebra $\mathbb C[ab]$
and the isomorphism classes of simple \emph{ungraded} $U(\mathfrak s)$-modules which are $H$-torsion free of rank $1$. More precisely, using the notation of Section~\ref{s3.2}, we have the following. 

\begin{theorem}\label{thm:s4-ungraded}
The modules $\Theta^*(F_u)=\Theta^*(D_u)$  form a complete and irredundant list of 
representatives of the isomorphism classes of simple
$U(\mathfrak s)$-modules  which are $H$-torsion free of rank $1$. 
\end{theorem}

\begin{proof}
By Lemma~\ref{lem:s4-ungraded-sigma0}, any such $U(\mathfrak s)$-module $M$ satisfies $\Sigma M=0$, hence factors through
$\overline U=U(\mathfrak s)/\langle\Sigma\rangle\cong \mathbf{A}_1$ by Proposition~\ref{prop:s4-Ubar-A1}.
Conversely, if $F$ is $\mathbf{A}_1$-module which is  torsion free of 
rank $1$ over $\mathbb C[ab]$, then $\Theta^*(F)$ is $H$-torsion free of rank $1$.
Indeed, $\Theta(\mathtt h)=-\frac12(ab+ba)=-ab+\frac12$, and hence
$H=\mathbb{C}[\mathtt h]$ identifies with $\mathbb{C}[ab]$.
The classification statement follows from the results in Section~\ref{s3.2}.
\end{proof}

\subsection{Simple \emph{graded} $U(\mathfrak s)$-modules of $H$-superrank $(1|1)$}\label{s4.3}

Let $M=M_{\bar0}\oplus M_{\bar1}$ be a $\mathbb{Z}_2$-graded $U(\mathfrak s)$-module which is $H$-torsion free of superrank $(1|1)$.
Since $\Sigma^2$ is central by Lemma~\ref{lem:s4-sigma}, it acts on $M$ by a scalar.
Choose $\lambda\in\mathbb{C}$ such that
\begin{equation}\label{eq:s4-lambda}
\Sigma^2=(\lambda+\tfrac12)^2\cdot \mathrm{Id}_M.
\end{equation}

Let $\Bbbk=\mathbb{C}(\mathtt h)$ and let $\sigma_1(\mathtt h)=\mathtt h-1$.
Set $R_1:=\Bbbk[x,x^{-1},\sigma_1]$.
For $u\in \Bbbk^{\times}$, let $N_u$ be the one-dimensional 
$R_1$-module on $\Bbbk$ with $x\cdot b=\sigma_1(b)\,u$.
Let $G_{\sigma_1}$ be the subgroup defined as in Section~\ref{s1.2} (with $\sigma$ replaced by $\sigma_1$).

For $u\in \Bbbk^{\times}$, define a $\mathbb{Z}_2$-graded $U(\mathfrak s)$-module $M_{u,\lambda}:=N_u\oplus N_u$ (even $\oplus$ odd) by
\begin{equation}\label{eq:s4-matrix}
\mathtt h\mapsto \mathtt h\,I_2,\qquad
\mathtt p\mapsto \begin{pmatrix}0&x\\ x&0\end{pmatrix},\qquad
\mathtt q\mapsto \begin{pmatrix}0&\frac{\mathtt h-\lambda}{2}\,x^{-1}\\[2pt]
\frac{\mathtt h+\lambda+1}{2}\,x^{-1}&0\end{pmatrix}.
\end{equation}

\begin{lemma}\label{lem:s4-matrix}
We have the following:
\begin{enumerate}[$($a$)$]
\item\label{lem:s4-matrix.1} The assignment \eqref{eq:s4-matrix} 
defines on $M_{u,\lambda}$ the structure of 
a $\mathbb{Z}_2$-graded $U(\mathfrak s)$-module.
\item\label{lem:s4-matrix.2} The element
$\Sigma$ acts on $M_{u,\lambda}$ as $(\lambda+\frac12)\mathrm{diag}(1,-1)$,
in particular, $\Sigma^2$ acts as in Formula~\eqref{eq:s4-lambda}.
\item\label{lem:s4-matrix.3}
For the $\mathfrak{sl}_2$-Casimir $\mathtt c$ from \eqref{eq1a} 
(with $\mathtt e=\mathtt p^2$ and $\mathtt f=-\mathtt q^2$) one has
\[
\mathtt c|_{(M_{u,\lambda})_{\bar0}}=(\lambda+1)^2,\qquad
\mathtt c|_{(M_{u,\lambda})_{\bar1}}=\lambda^2 .
\]
\end{enumerate}
\end{lemma}

\begin{proof}
The relations are checked by direct computations 
in $R_1$, using $x r=\sigma_1(r)x$ and $x^{-1}r=\sigma_1^{-1}(r)x^{-1}$, 
for $r\in \Bbbk$. The formula for $\Sigma$ follows from 
$\Sigma=\mathtt p\mathtt q-\mathtt q\mathtt p+\frac12$ also by direct computation.
This implies Claims~\eqref{lem:s4-matrix.1} and \eqref{lem:s4-matrix.2}.

To prove Claim~\eqref{lem:s4-matrix.3}, note that 
$\mathtt c=(\mathtt h+1)^2+4\mathtt f\mathtt e=(\mathtt h+1)^2-4\mathtt q^2\mathtt p^2$.
A direct computation gives that $\mathtt q^2\mathtt p^2$ acts by
$\frac{(\mathtt h-\lambda)(\mathtt h+\lambda+2)}{4}$ on $(M_{u,\lambda})_{\bar0}$
and by $\frac{(\mathtt h+\lambda+1)(\mathtt h-\lambda+1)}{4}$ on $(M_{u,\lambda})_{\bar1}$,
which yields the stated scalars.
\end{proof}

Let $\mathfrak s_{\bar0}\cong\mathfrak{sl}_2$ be the even part of $\mathfrak s$.
This is the subalgebra spanned by $\mathtt e=\mathtt p^2$,   $\mathtt f=-\mathtt q^2$, and $\mathtt h$.
Then $(M_{u,\lambda})_{\bar0}$ is a $U_{(\lambda+1)^2}$-module.
Let $L^{\bar0}_{u,\lambda}\subset (M_{u,\lambda})_{\bar0}$ be the simple socle of $(M_{u,\lambda})_{\bar0}$ as a $U_{(\lambda+1)^2}$-module
and set
\begin{equation}\label{eq:s4-Ldef}
L_{u,\lambda}:=U(\mathfrak s)\,L^{\bar0}_{u,\lambda}\ \subset\ M_{u,\lambda}.
\end{equation}

To be able to properly formulate the isomorphism criterion,
we introduce the following group:
\[
\Gtil:=\{\pm 1\}\,G_{\sigma_1}=\{\pm g\,:\,g\in G_{\sigma_1}\}\subset \Bbbk^{\times} .
\]

\begin{theorem}\label{thm:s4-graded}
\begin{enumerate}[$($a$)$]
\item \label{thm:s4-graded.1}
For all $u\in \Bbbk^{\times}$ and $\lambda\in\mathbb{C}$, 
the module $L_{u,\lambda}$ is a simple $\mathbb{Z}_2$-graded $U(\mathfrak s)$-module
which is $H$-torsion free of superrank $(1|1)$ and 
on which $\Sigma^2$ acts as in Formula~\eqref{eq:s4-lambda}.
\item \label{thm:s4-graded.2}
Every simple $\mathbb{Z}_2$-graded $U(\mathfrak s)$-module which is $H$-torsion free of superrank $(1|1)$
is isomorphic to $L_{u,\lambda}$, for some $u\in \Bbbk^{\times}$ and some $\lambda\in\mathbb{C}$.
\item \label{thm:s4-graded.3}
For $u,v\in \Bbbk^{\times}$ and $\lambda,\lambda'\in\mathbb C$, one has:
\begin{enumerate}[$($i$)$]
\item\label{thm:s4-graded.3-1} 
$L_{u,\lambda}\cong L_{v,\lambda'}$ as $\mathbb Z_2$-graded modules 
(i.e.\ via an even isomorphisms) if and only if
\[
\lambda=\lambda' \qquad\text{and}\qquad v\in \Gtil\cdot u .
\]
\item\label{thm:s4-graded.3-2} 
$L_{u,\lambda}\cong \Pi(L_{v,\lambda'})$ as $\mathbb Z_2$-graded modules if and only if
\[
\lambda'=-\lambda-1 \qquad\text{and}\qquad v\in \Gtil\cdot u .
\]
In particular, $\Pi(L_{u,\lambda})\cong L_{u,-\lambda-1}$.
\end{enumerate}
\end{enumerate}
\end{theorem}

\begin{proof}
By construction, $L_{u,\lambda}$ is a 
$\mathbb{Z}_2$-graded $U(\mathfrak s)$-submodule of $M_{u,\lambda}$.
Since the action of $x$ on $N_u$ is invertible, the operator
\[
\mathtt p=\begin{pmatrix}0&x\\ x&0\end{pmatrix}
\]
is invertible on $M_{u,\lambda}$. Hence, any nonzero graded $U(\mathfrak s)$-submodule $W\subset M_{u,\lambda}$ satisfies
$W_{\bar0}\neq 0$. This implies that 
$W_{\bar0}$ is a nonzero $U(\mathfrak s_{\bar0})$-submodule 
of $(M_{u,\lambda})_{\bar0}$ with central character
$(\lambda+1)^2$ and therefore contains the simple socle $L^{\bar0}_{u,\lambda}$.
It follows that we have 
$L_{u,\lambda}=U(\mathfrak s)\,L^{\bar0}_{u,\lambda}\subset W$, so $L_{u,\lambda}$ is the unique minimal nonzero graded submodule of
$M_{u,\lambda}$, hence simple. The $H$-torsion freeness and superrank $(1|1)$ are inherited from $M_{u,\lambda}$, and the fact that $\Sigma^2$ acts as in Formula~\eqref{eq:s4-lambda}
holds by Lemma~\ref{lem:s4-matrix}. This proves Claim~\eqref{thm:s4-graded.1}.
$\mathfrak{osp}(1|2)$
Conversely, let $M=M_{\bar0}\oplus M_{\bar1}$ be a simple graded $U(\mathfrak s)$-module which is $H$-torsion free of superrank $(1|1)$.
Localizing yields $M_{\Bbbk}:=\Bbbk\otimes_H M\simeq \Bbbk^{1|1}$.
Since $\Sigma^2$ is central, it acts on $M$ by a scalar.
Choose $\lambda \in \mathbb C$ such that $\Sigma^2$ acts as in 
Formula~\eqref{eq:s4-lambda}.
Choose $0\neq v_{\bar0}\in (M_{\Bbbk})_{\bar0}$ and set $v_{\bar1}:=\mathtt p\,v_{\bar0}$.
Then $v_{\bar1}\neq 0$ and $(v_{\bar0},v_{\bar1})$ is a homogeneous $\Bbbk$-basis of $M_{\Bbbk}$.
After an appropriate rescaling of this basis, 
$\mathtt p$ is represented by $\begin{pmatrix}0&x\\ x&0\end{pmatrix}$, for some invertible additive map $x:\Bbbk\to \Bbbk$ satisfying
\[
x(r\,b)=\sigma_1(r)\,x(b),\qquad \text{for all }r,b\in \Bbbk.
\] 
Then $x$ is uniquely determined by $u:=x(1)\in \Bbbk^{\times}$ and satisfies $x\cdot b=\sigma_1(b)\,u$, for all $b\in \Bbbk$.
Consequently $M_{\Bbbk}\simeq N_u\oplus N_u$ as an $R_1$-module. Using the relations $\mathtt p\mathtt q+\mathtt q\mathtt p=\mathtt h$ together with \eqref{eq:s4-lambda} (equivalently, the action of $\Sigma$),
one checks that $\mathtt q$ is forced to have the form in \eqref{eq:s4-matrix}. Thus $M_{\Bbbk}\simeq (M_{u,\lambda})_{\Bbbk}$ as graded $U(\mathfrak s)$-modules.
Now $M_{\bar0}$ is a torsion-free $U(\mathfrak s_{\bar0})$-module of rank $1$ with central character $(\lambda+1)^2$, hence embeds into
$(M_{u,\lambda})_{\bar0}$ and contains its socle $L^{\bar0}_{u,\lambda}$. Therefore $M$ contains
$U(\mathfrak s)L^{\bar0}_{u,\lambda}=L_{u,\lambda}$, and, by simplicity, $M\cong L_{u,\lambda}$. This proves Claim~\eqref{thm:s4-graded.2}.

Let us now prove Claim~\eqref{thm:s4-graded.3-1}.
Lemma~\ref{lem:s4-matrix} implies that,  on $L_{u,\lambda}$, we have
\[
\Sigma|_{(L_{u,\lambda})_{\bar0}}=(\lambda+\tfrac12)\,\mathrm{Id},\qquad
\Sigma|_{(L_{u,\lambda})_{\bar1}}=-(\lambda+\tfrac12)\,\mathrm{Id}.
\]
Hence any \emph{even} isomorphism $L_{u,\lambda}\to L_{v,\lambda'}$ forces $\lambda=\lambda'$.
Fix $\lambda=\lambda'$. If $v\in \Gtil\cdot u$, write $v=\varepsilon\,\frac{r}{\sigma_1(r)}u$ with $r\in \Bbbk^{\times}$ and $\varepsilon\in\{\pm 1\}$.
For $\varepsilon=1$, multiplication by $r$ gives an isomorphism $N_u\to N_v$ of $R_1$-modules, hence an isomorphism
$M_{u,\lambda}\to M_{v,\lambda}$ of graded $U(\mathfrak s)$-modules,
and it carries $L_{u,\lambda}$ onto $L_{v,\lambda}$.
For $\varepsilon=-1$, the same map gives an isomorphism $M_{u,\lambda}\to M_{w,\lambda}$ with $w=\frac{r}{\sigma_1(r)}u$, and, composing with the
even automorphism of $M_{w,\lambda}=N_w\oplus N_w$ given by $(a,b)\mapsto(a,-b)$, yields an isomorphism $M_{w,\lambda}\to M_{-w,\lambda}=M_{v,\lambda}$,
again sending $L_{u,\lambda}$ onto $L_{v,\lambda}$.

Conversely, assume $L_{u,\lambda}\cong L_{v,\lambda}$ as $\mathbb{Z}_2$-graded modules. 
Localizing, we obtain an isomorphism
$(L_{u,\lambda})_{\Bbbk}\cong (L_{v,\lambda})_{\Bbbk}$ of $\mathbb{Z}_2$-graded $U(\mathfrak s)$-modules.
Using the normal form \eqref{eq:s4-matrix}, both localized modules are isomorphic to $\Bbbk^{1|1}$ with
$\mathtt h=\mathtt h I_2$ and $\mathtt p=\begin{pmatrix}0&x_u\\ x_u&0\end{pmatrix}$ (resp.\ $\mathtt p=\begin{pmatrix}0&x_v\\ x_v&0\end{pmatrix}$),
where $x_u\cdot b=\sigma_1(b)u$ and $x_v\cdot b=\sigma_1(b)v$.
Let $\Phi$ be an even isomorphism between these localized modules. Since $\Phi$ commutes with $\mathtt h$, it is diagonal,
$\Phi=\mathrm{diag}(r_0,r_1)$, for some $r_0,r_1\in \Bbbk^{\times}$.
Commuting with $\mathtt p$ implies $r_1=\pm r_0$.
If $r_1=r_0$, then the relation $\Phi x_u = x_v \Phi$ yields
 $v \in G_{\sigma_1}\cdot u$, while, if $r_1=- r_0$, 
 we obtain  $v \in  (-G_{\sigma_1})\cdot u$. Thus $v\in \Gtil\cdot u$.

Finally, we prove Claim~\eqref{thm:s4-graded.3-2}.
Note that applying $\Pi$ swaps the two homogeneous components. Thus the eigenvalue of $\Sigma$ on the even component changes from
$\lambda+\tfrac12$ to $-(\lambda+\tfrac12)$, which corresponds to replacing $\lambda$ by $-\lambda-1$.
Concretely, conjugating the matrices in \eqref{eq:s4-matrix} by the permutation matrix
$P=\begin{pmatrix}0&1\\ 1&0\end{pmatrix}$ interchanges the two diagonal blocks and transforms the parameter $\lambda$ into $-\lambda-1$.
The condition on $u$ is unchanged, and the argument above shows that the parameter is determined up to $\Gtil$.
This proves Claim~\eqref{thm:s4-graded.3-2} and, 
in particular, that $\Pi(L_{u,\lambda})\cong L_{u,-\lambda-1}$.
\end{proof}

\begin{remark}\label{rem:s4-spanning}Writing $u^{(2)}:=u\,\sigma_1(u)$ and $\sigma=\sigma_1^2$ (so $\sigma(\mathtt h)=\mathtt h-2$), the restriction of $L_{u,\lambda}$
to the even subalgebra $\mathfrak s_{\bar0}\cong\mathfrak{sl}_2$ can be described as follows.
The action of $\mathtt e=\mathtt p^2$ on both homogeneous components is given by $x^2$, hence the $\sigma$-parameter is $u^{(2)}$
on both $(L_{u,\lambda})_{\bar0}$ and $(L_{u,\lambda})_{\bar1}$, while Lemma~\ref{lem:s4-matrix} gives the Casimir scalars
$(\lambda+1)^2$ and $\lambda^2$, respectively. Therefore, in the notation of Theorem~\ref{thm-main}, one has
\[
(L_{u,\lambda})_{\bar0}\cong K_{u^{(2)}},\qquad (L_{u,\lambda})_{\bar1}\cong K_{u^{(2)}},
\]
where $K_{u^{(2)}}$ is taken with central character $\vartheta=(\lambda+1)^2$ in the even part and $\vartheta=\lambda^2$ in the odd part.
Equivalently, we may write
\[
\mathrm{Res}^{U(\mathfrak s)}_{U(\mathfrak{sl}_2)}\,L_{u,\lambda}\ \cong\ K_{u^{(2)},(\lambda+1)^2}\ \oplus\ K_{u^{(2)},\lambda^2}.
\]
\end{remark}

We note that, for classification of simple $\mathfrak{osp}(1|2)$-modules, one could also 
use some of the ideas from \cite{CM21}. The approach from \cite{CM21}, however, is not applicable at all
for classification of ungraded modules, moreover, it is not applicable in its full strength
even in the graded case due to the abscence of the so-called {\em grading operator} in the terminology of \cite{CM21}. For example, Claim~\eqref{thm:s4-graded.3-2} of Theorem~\ref{thm:s4-graded}
is not derivable using the approach of \cite{CM21}.

\vspace{2mm}

\noindent
(D.~Grantcharov) Department of Mathematics, University of Texas at Arlington, USA\\
Email address: {\tt grandim\symbol{64}uta.edu}

\noindent
(L.~K{\v r}i{\v z}ka) Charles University, Prague, CZECH REPUBLIC\\
Email address: {\tt krizka.libor\symbol{64}gmail.com}

\noindent
(V.~Mazorchuk) Department of Mathematics, Uppsala University, Uppsala, SWEDEN \\
Email address: {\tt mazor\symbol{64}math.uu.se}

\end{document}